\documentclass[aop,MSNbibl,citesort,dvips]{arximspdf}

\doi{10.1214/10-AOP569}
\volume{39}
\issue{3}
\pubyear{2011}
\firstpage{1122}
\lastpage{1136}

\makeatletter

\newcommand{\overset}{\stackrel}

\newtheorem{theorem}{Theorem}
\newtheorem{corollary}[theorem]{Corollary}
\newtheorem{lemma}[theorem]{Lemma}

\newcommand{\RR}{\mathbb{R}}
\newcommand{\ZZ}{\mathbb{Z}}
\newcommand{\NN}{\mathbb{N}}
\newcommand{\half}{\frac{1}{2}}
\newcommand{\ra}{\rightarrow}
\newcommand{\lra}{\leftrightarrow}

\newcommand{\PATH}{\operatorname{PATH}}

\newcommand{\Prob}{\mathbf P}
\newcommand{\eff}{\operatorname{eff}}

\makeatother

\begin{document}
\begin{frontmatter}

\title{Cutpoints and resistance of random walk paths}
\runtitle{Cutpoints and resistance of random walk paths}

\begin{aug}
\author[A]{\fnms{Itai} \snm{Benjamini}\ead[label=e1]{itai.benjamini@weizmann.ac.il}},
\author[B]{\fnms{Ori} \snm{Gurel-Gurevich}\corref{}\ead[label=e2]{origurel@microsoft.com}} and
\author[B]{\fnms{Oded} \snm{Schramm}\ead[label=e3]{schramm@microsoft.com}}
\runauthor{I. Benjamini, O. Gurel-Gurevich and O. Schramm}
\affiliation{The Weizmann Institute of Science, Microsoft Research and
Microsoft Research}
\address[A]{I. Benjamini\\
Faculty of Mathematics\\
and Computer Science\\
The Weizmann Institute of Science\\
POB 26, Rehovot, 76100\\
Israel\\
\printead{e1}}
\address[B]{O. Gurel-Gurevich\\
O. Schramm\\
Microsoft Research\\
One Microsoft Way\\
Redmond, Washington 98052-6399\\
USA\\
\printead{e2}\\
\phantom{E-mail: }\printead*{e3}}
\end{aug}

\received{\smonth{6} \syear{2009}}
\revised{\smonth{5} \syear{2010}}

%
\begin{abstract}
We construct a bounded degree graph $G$, such that a simple random walk
on it is transient but the random walk path (i.e., the subgraph of all
the edges the random walk has crossed) has only finitely many
cutpoints, almost surely. We also prove that the expected number of
cutpoints of any transient Markov chain is infinite. This answers two
questions of James, Lyons and Peres [\textit{A~Transient Markov
Chain With Finitely Many Cutpoints} (2007) Festschrift for David
Freedman].

Additionally, we consider a simple random walk on a finite connected
graph $G$ that starts at some fixed vertex $x$ and is stopped when it
first visits some other fixed vertex $y$. We provide a lower bound on
the expected effective resistance between $x$ and $y$ in the path of
the walk, giving a partial answer to a question raised in [\textit{Ann.
Probab.} \textbf{35} (2007) 732--738].
\end{abstract}

%

\end{frontmatter}

\section{Introduction}

In this paper, we study natural geometric and potential theoretic
properties of the simple random walk path on general graphs. Given a
graph~$G$, a \textit{simple random walk} on $G$ is a Markov chain, $\{
X_t\}_{t=0}^\infty$, on
the vertices of the graph, such that the distribution of $X_{t+1}$
given the current state $X_t$, is uniform among the neighbors of
$X_t$. Given a sample of the simple random walk, the \textit{path}
of the walk (denoted $\PATH$) is the subgraph consisting of all the
vertices visited
and edges traversed by the walk.

Given a rooted graph $(G,g_0)$, a vertex $x$ of $G$ is a
\textit{cutpoint} if it separates the root $g_0$ from infinity, that is,
if removing $x$ from $G$ would result in $g_0$ being in a finite
connected component. A vertex is a cutpoint of the path of a walk if it
is a cutpoint of $(\PATH, X_0)$.

In \cite{path,trace} it was shown that the path of a simple random
walk is always a \textit{recurrent} graph, that is, a simple random walk
on the path returns to the origin, almost surely. If $G$ is of
bounded degree and the path has infinitely many cutpoints, then the
path is obviously recurrent. Indeed, this is the case when $G$ is the
Euclidean lattice, as shown in \cite{l,jp}. The question arises
naturally: does the path of a simple random walk on every graph have
infinitely many cutpoints, almost surely?

This question was raised in \cite{JLP}, where an
example of a nearest neighbor random walk on the integers that has
only finitely many cut-times almost surely is provided. A \textit
{cut-time} is a
time $t$ such that the past of the walk $\{X_0,\ldots,X_t\}$ is
disjoint from its future $\{X_{t+1},\ldots\}$. Clearly, a cut-time
$t$ induces a cutpoint $X_t$, but not vice verse. Indeed, in the
example in \cite{JLP}, the path of the walk is simply the integers,
and so every vertex (but 0) is a cutpoint. Moreover, \cite{JLP} left
open the question of whether there is such a \textit{simple random walk}
on a bounded degree graph.

Returning to our question, we answer it in the negative.
\begin{theorem} \label{example}
There exists a bounded degree graph $G$ such that the path of the
simple random walk on $G$ has finitely many cutpoints, almost
surely.
\end{theorem}

In Section \ref{example_proof}, we construct an ad hoc example and
prove it has the claimed property. In Section \ref{horns}, we argue that
subgraphs of $\ZZ^d$, $d\ge3$, spanned by vertices satisfying
$x_1\le f(x_2,\ldots,x_n)$ for an appropriate choice of $f$
also exhibit this property.

In \cite{JLP} it is noted that in their example (as well as in similar
examples in \cite{CFR}), the \textit{expected}
number of cut-times is infinite. We show that this is, in fact, the
case for any transient Markov chain.
\begin{theorem} \label{infinity}
For every transient Markov chain
the expected number of cut-times (and hence cutpoints) of the path
is infinite.
\end{theorem}

Lastly, we consider the resistance of the path when considered as an
electrical network with each edge being a unit resistor. As mentioned,
in \cite{path,trace} it is proved that the
path of a simple random walk is recurrent, almost surely, and therefore
its resistance to
infinity is infinite. In Section \ref{resistance} we give a
quantitative version of this theorem, providing explicit bounds on the
resistance of a finite
portion of the path, in terms of the maximal degree of $G$ and the
probability of return from the boundary of the finite portion to the origin.

\section{Open questions}

Some open questions present themselves: under what conditions does the
path of a random walk have a.s. infinitely many cutpoints? This
question was largely resolved in \cite{CFR} for the special case of
nearest neighbors walks on the integers. For general (bounded degree)
graphs, we find the following 2 questions interesting:
\begin{itemize}
\item
Does a strictly positive $\liminf$ speed of a simple random walk imply
having a.s. infinitely many cut points of its path?
\item
Does the path of a simple random walk on any transient vertex
transitive graph have a.s. infinitely many cut points?
\end{itemize}

We conjecture that the answer to both questions is positive.

Theorem \ref{example} can be easily generalized to show that for every
positive integer $k$, the path in our example has only finitely many
minimal \textit{cutsets} of size $k$ (i.e., sets whose removal from the
path disconnect $X_0$ from infinity). This is done by choosing a
suitably large $M$ in the proof. Since the construction itself does not
depend on $M$, we know that there are actually only finitely many
finite minimal cutsets in the path. Furthermore, by allowing $M$ to
depend of the layer and slowly tend to infinity, one can get an
explicit lower bound on the rooted isoperimetric profile of the path.
It is well known (see, e.g., \cite{T}) that any graph satisfying a
large enough rooted isopermietric inequality is transient. In our
context the natural question is this:

\begin{itemize}
\item
Given an isoperimetric profile $f$ which does not imply transience, is
there a bounded degree graph $G_f$, such that the path of a simple
random walk on $G_f$ satisfies the rooted $f$-isoperimetric inequality?
In other words, is there some upper bound on the isoperimetric profile
of the path?
\end{itemize}

\section{\texorpdfstring{Proof of Theorem \protect\ref{example}}{Proof of Theorem 1}}
\label{example_proof}

\subsection{Construction}

Let $E_n$ be a sequence of $d$-regular expanders, where $E_n$ has
$n$ vertices. The graph we describe is composed of layers, $G_j$ for
$j\in\NN$, where edges are only within a single layer or between
adjacent layers. Fix some $\alpha>1$. For $2^k / k^\alpha\le j <
2^{k+1} / (k+1)^\alpha$, we let $G_j$ be a copy of $E_{2^k}$.
(Actually, this only defines $G_j$ for $j\ge j_0$ for some $j_0\in\NN$,
which depends on $\alpha$. For $j\in\NN\cap[0,j_0)$, we take $G_j=G_{j_0}$.)
If $G_j$ and $G_{j+1}$ are of the same size, we connect $x\in G_j$ with
$y \in G_{j+1}$ if $x$ and $y$ are connected in $E_{2^k}$. If $G_{j+1}$
is twice the size of $G_j$, we choose some bipartite graph on the
vertices of $G_j \cup G_{j+1}$ which has $2d$ edges attached to each
vertex in $G_j$ and $d$ edges attached to each vertex in $G_{j+1}$.
Denote the resulting graph $G$. We claim that this $G$
has the properties we seek in Theorem~\ref{example}.

\subsection{Proof}

Let $Z_t=(X_t,Y_t)$ be a simple random walk on $G$, where $X_t$
marks the layer, and $Y_t$ the location in $V_{X_t}$
(here, $V_x$ denotes the set of vertices at layer~$x$). Since the
expanders are of constant degree, the probability of the walk
moving up, down or staying in the same layer is independent of the
position inside the layer. Therefore, $X_t$ is a (lazy)
random walk on $\NN$, which can be easily described as follows.
Let $w(j,j+1)=w(j+1,j)$ denote the number of edges connecting $G_j$ and
$G_{j+1}$,
and let $w(j,j)$ be twice the number of edges of $G_j$.
Then $X_t$ is the network random walk on the network $(\NN,w)$, and
$\eta_{X_t}$ is a martingale,
where $\eta_j:=\sum_{i=j}^\infty r_i$ and $r_i:=1/w(i,i+1)$.
In particular, the probability
that such a walk starting from $j$ ever returns to $0$ is $\eta_j/\eta_0$.
Since $r_j \asymp j^{-1} \log^{-\alpha}(j)$,
where $\asymp$ means that the ratio is bounded
and bounded away from zero, we have $\eta_j\asymp\log^{1-\alpha}(j) $,
and $X_t$ is transient.

The Markov chain $X_t$ is the kind of chain which is given in \cite
{JLP} as an example
of a Markov chain with a.s. only finitely many cut-times. We will
analyze the walk more thoroughly in the following.

Fix some $0<\beta<1$ and $j\in\NN_+$.
Write
\begin{eqnarray*}
j_- &:=&\lfloor j-j^\beta\rfloor,\\
j_+ &:=&\lceil j+j^\beta\rceil.
\end{eqnarray*}
Define $s_0=s_0(j):=\inf\{t\in\NN\dvtx X_t=j_-\}$, $t_0=t_0(j):=\inf\{
t\in\NN\dvtx X_t=j_+\}$
and inductively $s_i=s_i(j):=\inf\{t>t_{i-1}\dvtx X_t=j_-\}$ and
$t_i=t_i(j):=\inf\{t>s_i\dvtx X_t=j_+\}$.
(As usual, the convention $\inf\varnothing=\infty$ is used.)
The \textit{linking} of $j$ is defined as $\ell(j):=\sup\{i\in\NN
\dvtx t_i<\infty\}$.
We fix some constant $M\in\NN_+$, and say that $j$ is \textit{linked}
if $\ell(j)\ge M$.
Let $I_j$ be the event that $j$ is not linked, and let $p_j:=\Prob(I_j)$.
\begin{lemma} \label{linked}
Almost surely, the set of $j\in\NN$ that are not linked is finite.
\end{lemma}
\begin{pf}
Let $p_j$ be the probability that $j$ is not linked.
When the walk is
at $j_+$, the probability of it never reaching $j_-$
again is
\[
1-\frac{\eta_{j_+}}{\eta_{j_-}}=\frac1{\eta_{j_-}}\sum
_{i=j_-}^{j_--1} r_i.
\]
Since
$\eta_j \asymp\log^{1-\alpha}(j)$ and $r_i \asymp j^{-1}
\log^{-\alpha}(j)$ for any $i\in\{j_-,\ldots,j_+\}$, we
get that this probability is $\asymp j^{\beta-1}\log^{-1}(j)$. Thus,
\[
p_j\asymp M j^{\beta-1}\log^{-1}(j) \asymp
j^{\beta-1}\log^{-1}(j) ,
\]
since $M$ is constant.

We would like to
estimate $\Prob(I_i \mid I_j)$ for $i<j$ (or more
precisely some variant thereof). For technical reasons we
impose the condition $i < j_-$.

Note that $I_j$ depends only on those steps of the walk between a
visit to $j_+$ and the next visit to $j_-$, if it
occurs. Therefore, the rest of the walk, that is, between visits to
$j_-$ and $j_+$, as well as before the first visit to
$j_+$, retains the law of the network walk when
conditioning on $I_j$.
Let $Q=Q(j)$ denote the segments of
the path between visits to $j_+$ and visits to $j_-$.
More precisely, the $k$th segment is
\[
Q^k=Q^k(j):= (X_{t_{k}},X_{t_{k}+1},\ldots,X_{s_k})
\]
for $k\in\{0,1,\ldots,\ell(j)-1\}$,
\[
Q^{\ell}:=(X_{t_\ell},X_{t_\ell+1},\ldots)
\]
for $k=\ell=\ell(j)$, and finally
\[
Q=Q(j):=(Q^1,Q^2,\ldots,Q^\ell) .
\]

Now, when the network walk is started at $j_-$ the
probability that it hits $j_+$ before $i_-$ is at least
some constant $c>0$ (because $i<j _-$). The probability of the
walk started at $i_+$ to hit $j_+$ before
$i_-$ is
\[
\frac{\eta_{i_+}-\eta_{i_-}}{\eta_{j_+}-\eta_{i_-}}
\asymp\frac{ i^\beta}{ j-i} .
\]
Thus, the conditional independence noted above implies that
when $i<j_-$ on the event $\ell(j)<M$ we have
%
%
\begin{equation}
\label{e.Ii}
\Prob(I_i\mid Q(j))\ge
\Prob\bigl(\ell(i)=0 \mid Q(j)\bigr)
\asymp c^{\ell(j)} \frac{i^\beta}{ j-i}
\asymp\frac{ i^\beta}{j-i} ,
\end{equation}
where the implied and explicit constants may depend on
$\alpha$, $\beta$ and $M$.

Let $A_k=\sum_{2^k<j\le2^{k+1}} 1_{I_j}$. For $2^k<j\le2^{k+1}$ we
have $p_j \asymp j^{\beta-1}\log^{-1}(j) \asymp2^{k\beta- k} / k$.
Therefore, $E(A_k) = \sum_{2^k<j\le2^{k+1}} p_j \asymp
2^{k\beta}/k$. Also, $E(A_{k-1}+A_k) \asymp2^{k\beta}/k$.

Next, we would like to bound $E(A_{k-1} + A_k \mid A_k >0)$. If $A_k>0$
then $I_j$ occurs, for some $2^k<j\le2^{k+1}$. Let $j^*$ be the
largest of this set; that is, $j^*:=\max\{j\in(2^k,2^{k+1}]\dvtx
I_j\mbox
{ holds}\}$.
Note that $j^*=j$ is $Q(j)$-measurable. Therefore,
\begin{eqnarray*}
E(A_{k-1}+A_k \mid A_k >0)
&\ge&
\min_{2^k<j\le2^{k+1}} E(A_{k-1}+A_k \mid j^*=j, A_k>0)
\\
&\ge&
\min_{2^k<j\le2^{k+1}}
\inf_{z}
\sum_{i=2^{k-1}}^{j_-}
\Prob\bigl(I_i \mid Q(j)=z\bigr) ,
\end{eqnarray*}
where the infimum is over all possible $z$ such that
$\{Q(j)=z\}\cap\{j^*=j\}$ is possible.
Thus, (\ref{e.Ii}) gives
\[
E(A_{k-1}+A_k \mid A_k >0)
\ge c
\min_{2^k<j\le2^{k+1}}
\sum_{i=2^{k-1}}^{j_-} \frac{i^\beta}{j-i}\asymp
2^{k\beta} k (1-\beta)\asymp2^{k\beta} k .
\]
Therefore, $\Prob(A_k>0)=E(A_{k-1}+A_k)/E(A_{k-1}+A_k \mid A_k >0)
\asymp1/k^2$. Thus, $\sum_{k=1}^\infty\Prob(A_k>0) <\infty$, which
implies that a.s. at most finitely many $k$
satisfy \mbox{$A_k>0$}.
\end{pf}

Returning to the full random walk $Z_t$ we prove that if $j$ is a
linked point (vertex) of the walk $X_t$ then the probability of any
point in
$V_j$ being a cutpoint of $Z_t$ is small (for suitable $\beta$ and
$M$).

Fix $j$ and first assume for simplicity that there is no $k$ such that
$j_- \le2^k / k^\alpha\le j_+$; then $X_t$ is a martingale
in this range.
Call a
segment of the random walk timeline, $s, s+1,\ldots,t$, a \textit{pass
around $j$} if $X_s=j_-$, $X_t=j_+$ and $X_i$ is neither
$j_-$ nor $j_+$ for $i=s+1,\ldots,t-1$. In other words,
in a pass, the walk starts at $j_-$ and ends at $j_+$,
all the while staying between these two endpoints. If $j$ is linked
then there are at least $M$ (time-)disjoint passes around it. Note that
we might as well have used the passes in the reverse direction (from
$j_+$ to $j_-$), getting $2M$ passes, but since
$M$ is arbitrary, there is no need for this.

If $s,\ldots,t$ is a pass around $j$, then $X_s,\ldots,X_t$ is a
delayed simple random walk on $\NN$, started at $j_-$ and conditioned
on hitting $j_+$ before returning to $j_-$. Next, we
prove some simple facts about the typical behavior of such a walk.

\subsection{Interlude: Two elementary facts about SRW}

Let $x_0,x_1,\ldots$ be a simple random walk on $\ZZ$, started at
$x_0=0$. Let $\tau_i=\min\{t>0\mid x_t=i\}$ be the hitting time of $i$
(excluding the starting position). Let $a\in\NN_+$.
We are interested in the
behavior of the walk conditioned on $\tau_a<\tau_0$.
\begin{lemma} \label{chernoff}
\[
\Prob(\tau_a < t \mid\tau_a < \tau_0) < 2 a t e^{-a^2/4t}.
\]
\end{lemma}
\begin{pf}
For any $s\le t$, a Chernoff bound yields $\Prob(x_s \ge a) \le2
e^{-a^2/4s} \le2 e^{-a^2/4t}$. By a union bound, $\Prob(\tau_a < t)
\le2 t e^{-a^2/4t}$. Since we condition on an event of probability
$\Prob(\tau_a < \tau_0)=1/a$, the conditional probability cannot
increase by more than a factor of $a$.
\end{pf}

Note: this is far from the best bound, but it suffices for our
purposes. Using the reflection principle and the central limit
theorem one can get a bound of the form $ C_\varepsilon
e^{-a^2/(2+\varepsilon)t}$ for any $\varepsilon>0$, if not better.

Assume, for simplicity, that $a$ is even and let $b=a/2$. Let $B=\{
t<\tau_a \dvtx x_t = b\}$, that is, the set of times where the walk visits
$b$ before hitting $a$.
\begin{lemma} \label{exp} For every $m\in\NN$,
\[
\Prob(|B| > m \mid \tau_a < \tau_0) < 2 e^{-2m/a}.
\]
\end{lemma}
\begin{pf}
First, condition on $\tau_b < \tau_0$. Every time the walk visits
$b$, there is probability of $1/(b-1)$ that the walk
never returns to $b$ before hitting $\{0,a\}$. Therefore,
\[
\Prob(|B| > m, \tau_a<\tau_0 \mid \tau_b < \tau_0)
\le\biggl(1-\frac1{b-1}\biggr)^m
< \biggl(1-\frac2a\biggr)^m \le
e^{-2m/a}.
\]
Since $\Prob(\tau_a < \tau_0 \mid\tau_b < \tau_0) = 1/2$,
we get the extra factor of 2 in our bound.
\end{pf}

Note: these two lemmas apply also to lazy simple random walks. In
Lem\-ma~\ref{chernoff}, laziness only improves the bound, as it takes
longer to reach $a$. [One has to account for the change
in $\Prob(\tau_a<\tau_0)$, but this is rather minor.]
In Lemma \ref{exp}, the bound changes to $2
e^{-2Ck/a}$ with $C$ depending on the probability to stay in
place.

\subsection{Proof, continued}

Returning to our original setup, we use Lemma \ref{chernoff} to show
that different passes around $j$ tend to intersect each other.
We continue to assume that there is no integer $k$ such that
$j_-\le2^k/k^\alpha\le j_+$.
\begin{lemma} \label{intersect}
Let $\mathcal A_j(s)$ be the event that there is a pass
around $j$ starting at time $s$, and on $\mathcal A_j(s)$ let
$\tau$ be the final time of the pass.
Let $\{v_i \dvtx
i=j_-,\ldots,j-1\}$ be arbitrary points in $G$, where $v_i\in
V_i$. Then
\[
\Prob\bigl(\{Z_s,\ldots,Z_\tau\} \cap\{v_i \dvtx i=j_-,\ldots,j-1\} =
\varnothing\mid \mathcal A_j(s) \bigr) < C e^{-j^{\beta-{1/2}
-\varepsilon}}
\]
holds for any $\varepsilon>0$ and some $C$ depending on $\varepsilon$.
\end{lemma}
\begin{pf}
Consider only the part of the pass until the first time
$\tau'\ge s$ when it it first hits $V_j$.
By Lemma \ref{chernoff} we get that the conditional
probability [given $\mathcal A_j(s)$] that $\tau'-s < j^{\beta+1/2}$
is at most $O(1) j^{2 \beta
+1/2} \exp(-{j^{\beta- 1/2}/4})$.

In the time range $t\in\{s,s+1,\ldots,\tau\}$, the
walk $Y_t$ is a simple random walk on $E_{2^k}$, where,
by assumption $2^k/k^\alpha
\le j_- < j <j_+ \le2^{k+1}/(k+1)^\alpha$. By the mixing property of
the expanders we chose, there is some $C>0$ such that the distribution
of the walk after $C k \asymp C \log j$ steps is $j^{-2}\asymp
2^{-2k}$-close (in total variation) to uniform. Therefore, the
probability of being at any specific vertex is at least $\half
2^{-k}$. This holds conditional on the entire history of the walk
except for the last $C \log j$ steps.

Therefore, every $C \log j$ steps the walk has a probability of at
least $2^{-k-1}$ of intersecting $\{v_i \mid i=j_-,\ldots,j-1\}$
(conditional on $X_t$ to be between $j_-$ and $j$ in this range).
Thus, the probability of not intersecting $\{v_i \mid
i=j_-,\ldots,j-1\}$ until time $j^{\beta+1/2}$ is bounded by
$(1-2^{-k-1})^{j^{\beta+1/2}/ C \log j}$. Since $2^{k} \asymp j
\log^\alpha j$, we get a bound of $O(1) \exp(-{j^{\beta- 1/2}/C
\log^{\alpha+1} j})$.

Both this\vspace*{1pt} probability and $\Prob(\tau'-s < j^{\beta+1/2})$ are
asymptotically smaller then $\exp(-{j^{\beta- 1/2 - \varepsilon}})$. Thus,
we get the required bound.
\end{pf}

The same conclusion also applies to a set of points $\{v_i \mid
i=j+1,\ldots,j_+\}$ on the other side of $j$. Let
$\tau_j=\min\{t \mid X_t=j\}$ and $\sigma_j=\max\{t\mid X_t=j\}$ be the
first and last visits to $V_j$.
\begin{corollary}
Conditional on $\{Z_0,\ldots,Z_{\tau_j}\}$ and
$\{Z_{\sigma_j},\ldots\}$ an independent pass around $j$ intersects both
with probability at least $1-C e^{j^{\beta-{1/2}-\varepsilon}}$.
\end{corollary}
\begin{pf}
These two sets each contain at least one element of each $V_i$ for
$i=j_-,\ldots,j-1$ and $i=j+1,\ldots,j_+$.
\end{pf}

To conclude the proof, we just need to show, using Lemma \ref{exp},
that the probability of the random walk to hit a specific point
during a pass is low.
\begin{lemma}\label{pass_through}
Let $v$ be an arbitrary point in $V_j$. With the notation of
Lemma \ref{intersect}, we have
\[
\Prob\bigl(v \in\{Z_s,\ldots,Z_\tau\}\mid \mathcal A_j(s), Z_s
\bigr) < C j^{\beta-1}
\]
for some constant $C$.
\end{lemma}
\begin{pf}
Let $B=\{t_1<\cdots<t_m\}$ be the set of times between $s$ and $\tau$
that the walk is in $V_j$. By Lemma \ref{exp}, we have
%
%
\begin{equation}\label{e.m}
\Prob(m > C_1 j^\beta
\log j ) < 2 j^{-2C_1}.
\end{equation}
Obviously, $t_i - s \ge j^\beta$ for any $i$, that is, the random walk
took at least $j^\beta$ steps before reaching $V_j$. By the mixing
property of the expanders we chose, there is some $C_2>0$ such that the
distribution on $Y_{s+j^\beta}$, conditioned on the history until time
$s$, is $e^{-C_2 j^\beta}$-close (in total variation) to uniform.
Since the distance to the uniform distribution can only decrease, we
have, for any $i$
\[
\Prob\bigl(Z_{t_i}=v \mid \mathcal A_j(s), Z_s \bigr) < |V_j|^{-1} +
e^{-C_2 j^\beta} < C_3 j^{-1}\log^{-\alpha} j.
\]
Combining with (\ref{e.m}) yields
\begin{eqnarray*}
\Prob\bigl(v\in B \mid \mathcal A_j(s), Z_s \bigr) &\le& \Prob(m >
C_1 j^\beta\log j ) + \sum_{i=1}^{C_1 j^\beta\log j} \Prob
\bigl(Z_{t_i}=v \mid \mathcal A_j(s), Z_s \bigr)
\\
&\le&2 j^{-2C_1} + C_1 j^\beta\log j C_3 j^{-1}\log^{-\alpha}j \le C
j^{\beta-1}
\end{eqnarray*}
for a proper choice of $C_1$.
\end{pf}

We now argue that our above conclusions also apply when there is some
$k\in\NN$ satisfying $j_-\le2^k/k^{\alpha}\le j_+$. For $j$ large,
there is clearly at most one such $k$. Let $\widetilde j$ be the value
of $\lfloor2^k/k^\alpha\rfloor$, that is, between $j_-$ and $j_+$. The
argument used in the proof of Lemma \ref{intersect} can just be applied
to the set $\{v_i\dvtx i_0\le i\le i_1\}$, where $j_-\le i_0\le i_1\le
j-1$, $i_1-i_0$ is proportional to $j^\beta$ and $\widetilde
j\notin[i_0,i_1]$. The next issue is that $X_t$ does not behave like a
martingale when in the range\vspace*{1pt} $[j_-,j_+]$. However, if we
define $g(i)=i$ for $i\le\widetilde j$ and $g(i)=\widetilde
j+(i-\widetilde j)/2$ for $i\ge \widetilde j$, then $g(X_t)$ behaves as
a martingale while $X_t\in[j_-,j_+]$, and the analogue of Lemma
\ref{exp} holds with easy modifications to the proof. Finally, it is
easy to adapt the proof Lemma \ref{pass_through} as well. The crucial
point here is that the edges connecting $V_{\widetilde j}$ and
$V_{\widetilde j +1}$ maintain the uniform distribution. In other
words, as the random walk passes from $V_{\widetilde j}$ to
$V_{\widetilde j +1}$ (or vice verse) its distribution can only get
closer to uniform. Therefore, we can safely ignore the steps between
these layers when calculating the distance to uniform. Since there are
plenty of steps to spare, the analysis remains valid.

Putting it all together we get:
\begin{corollary}
If $j$ is linked and $v\in V_j$, then
\[
\Prob( v \mbox{ is a cutpoint}) < C j^{M(\beta-1)}.
\]
\end{corollary}
\begin{pf}
Each pass around $j$ connects $\{Z_0,\ldots,Z_{\tau_j}\}$ and
$\{Z_{\sigma_j},\ldots\}$ without passing through $v$ with
probability at least $1-C j^{\beta-1}$, regardless of the history of the
walk. Thus, the probability that every one of the $M$ passes fails
to do so is bounded by $C j^{M(\beta-1)}$.
\end{pf}

Now, for $\half< \beta< 1$ and $M>2/(1-\beta)+2$, the expected
number of cutpoints in any $V_j$ for linked $j$ is finite. Since all
but finitely many layers are linked, the theorem is proved.

\section{Other graphs with finitely many cutpoints} \label{horns}

The examples provided by Theorem \ref{example} are perhaps not the most
natural ones. Are there simpler examples exhibiting this phenomenon?

There are. In fact, we claim that a suitably chosen subgraph of
$\ZZ^d$, for $d\ge3$, is such an example. Given a function
$f\dvtx\RR^+\ra\RR^+$, define the \textit{horn} of $f$ in $\ZZ^d$ to be
\[
H^d_f = \{(x_1,x_2,\ldots,x_d)\in\ZZ^d ; x_1 \ge0 ,
x_2^2+\cdots+x_d^2 \le f^2(x_1) \} .
\]

That is, the part of the positive half space where the distance to the
$x_1$-axis is less then $f(x_1)$. Taking $f=\sqrt[d-1]{x
\log^\alpha(x)}$, for $\alpha>1$, we get a ``barely transient'' graph,
similar to our original construction. The layers in this graph are
sets\vspace*{1pt} of points having the same $x_1$ coordinate. The
size\vspace*{1pt} of the $i$th layer is roughly $ f^{d-1}(i)= i
\log^\alpha(i)$. Standard arguments can be used to construct a flow in
$H^d_f$ from the origin to infinity having finite energy, thus showing
that this graph is transient.

The difference\vspace*{1pt} between $H^d_f$ and our previous example is twofold:
the layers are connected differently, and the layers themselves are
obviously not expanders, but some subset of $\ZZ^{d-1}$ instead.
Below is an outline of how to deal with these differences.

First, since the layers are not even regular, we cannot separate the
horizontal movement (along the $x_1$ axis) from the vertical (all
other directions). In order to prove Lemma \ref{linked} in this
case, one has to give some bounds on the minimal and maximal
probability of escape from layer $j$ (minimal and maximal w.r.t. the
location inside the layer). The argument of Lemma \ref{linked} is
rather robust, so the proof should be adaptable.

Second, since the layers are not expanders, the walk on them does
not mix as rapidly, which interferes with the proof of
Lemma \ref{intersect}. The
mixing time of layer $i$ in $H^d_f$ is of order $ f^2(i)= (i
\log^\alpha(i) )^{2/(d-1)}$. If $d\ge4$, then this is less then
$i^{2/3+\varepsilon}$. Going through the proof of Lemma \ref{intersect}
we see that one can get a bound of order $ e^{-j^{\beta-5/6-\varepsilon
}}$ in
this case, which is enough to proceed with the rest of the proof
when $5/6<\beta$.

What about $d=3$? The proof as written does not work since the
mixing time of layer $i$ is now more then $i$. However, the proof of
Lemma \ref{intersect} did not use the mixing of our random walk
optimally. We only sampled the walk once every mixing time steps and
ignored the rest of the steps. For $d=3$, one needs to improve on
that by first proving that if we have an $n\times n \times n$ cube,
consisting of $n$ layers, with at least one marked vertex in each
layer, then the probability of a simple random walk, started somewhere
in the middle layer,
to visit one of the marked vertices before reaching the first or last
layer, decays only logarithmically in $n$.\vspace*{1pt}

Since\vspace*{1pt} layer $i$ is roughly $\sqrt{i \log^\alpha(i)}$ by
$\sqrt{i \log^\alpha(i)}$, and the pass length
is $i^\beta$, which we may take to be bigger than
$\sqrt{i\log^\alpha(i)}$, the random walk would have more than
$i^{\beta- {1/2}- \varepsilon}$ opportunities to intersect the marked
vertices (i.e., previous passes), which yields an exponentially small
probability of failing to do so.
Of course, to prove this in full detail, one would have to also deal
with the behavior of the walk near the boundary of the layers, which
definitely would add significant complications. We do not pursue this here.

\section{\texorpdfstring{Proof of Theorem \protect\ref{infinity}}{Proof of Theorem 2}}

Next, we prove that even though the number of cutpoints can be
finite a.s., its expectation is always infinite. This is true for
any transient Markov chain, not necessarily reversible.

Let $X_i$ be a transient Markov chain, $S$ its state space and $T$
the transition probability matrix.

Define $f(s)$ to be the probability that a chain with the same law,
started at $s$, will ever visit $X_0$ (the starting state of $X$).
This function is harmonic for all $s\neq X_0$. Therefore, $f(X_i)$
is a martingale, as long as $X_i\neq X_0$.

First, we deal with the special case when the chain is irreducible.
In that case, $f$ is positive everywhere, that is, there is a positive
probability of returning to $X_0$ from any vertex.

The chain is transient, thus $\lim_{i\ra\infty} f(X_i) =0$, almost
surely. Let $M_n$ be the sequence of minima of $f(X_i)$ and $i_n$
the times in which these minima are achieved. More precisely,
$i_{n+1}=\min\{i \mid i>i_n , f(X_i)<f(X_{i_n})\}$ and
$M_n=f(X_{i_n})$. This sequence is infinite since $f(X_i)>0$ due to
irreducibility.

Given $i_{n-1}$ and $i_n$, let $j_n=\min\{j \mid j>i_n , f(X_j)\ge
M_{n-1}\}$, which is the first time $j\ge i_n$ at which the value of $f(X_j)$
exceeds the previously obtained minimum, or infinity if this never happens.
Note that $j_n$ is a stopping time. By applying the optional
stopping theorem, together with the positivity of $f$, we get
$E(f(X_{j_n}) \mid M_n) \le M_n$, where we take
$f(X_\infty)=\lim_{j\ra\infty} f(X_j)=0$. By definition,
$f(X_{j_n})\ge M_{n-1} > M_n$ if $j_n < \infty$. Therefore, $P(j_n <
\infty\mid M_n , M_{n-1}) \le\frac{M_n}{M_{n-1}}$.

Notice that if $j_n = \infty$ then $i_n$ must be a cut-time (and
$X_{i_n}$ a cutpoint), since $f(X_i)\ge M_{n-1}$ for $i<i_n$ and
$f(X_i)<M_n$ for $i\ge i_n$. Thus, given $M_{n-1}$ and $M_n$ the
probability that $i_n$ is a cut-time is at least
$1-\frac{M_n}{M_{n-1}}$.

Recall that $M_n$ is a monotone decreasing sequence, tending to 0.
For any such sequence, we have $\sum_{n=1}^\infty
(1-\frac{M_n}{M_{n-1}}) = \infty$, since $\prod_{n=1}^\infty
\frac{M_n}{M_{n-1}} = 0$. Putting it all together, we get
\begin{eqnarray*}
\sum_{n=1}^\infty P(i_n \mbox{ is a cut-time})
&=& \sum_{n=1}^\infty E\bigl(P(X_{i_n} \mbox{ is a cut-time} \mid
M_n,M_{n-1})\bigr)
\\
&\ge&\sum_{n=1}^\infty E\biggl(1-\frac{M_n}{M_{n-1}}\biggr) = E
\Biggl(\sum_{n=1}^\infty\biggl(1-\frac{M_n}{M_{n-1}}\biggr)\Biggr) \\
&=& \infty.
\end{eqnarray*}

What happens if our chain is not irreducible?
In that case the state
space can be decomposed into irreducible components. These are
equivalence classes of the equivalence relation consisting of pairs $(x,y)$
for which one can get from $x$ to $y$ with positive probability (in
possibly more
than one step), and one can get from $y$ to $x$ with positive probability.

If there is positive probability that the chain eventually stays in
some fixed
equivalence class $S$, then we may consider for some $x\in S$ the
probability to get to~$x$, and the previous proof applies to show
that the expected number of cut-times is infinite. Otherwise
the number of cut-times is infinite almost surely, because each transition
into a new equivalence class is necessarily a cut time.

\section{Bounding the resistance of the path} \label{resistance}

Even though the path of a simple random walk might have only
finitely many cutpoints, it is a recurrent subgraph of~$G$, as shown
in \cite{trace}. In other words, the resistance of the path, from
any vertex to infinity, is infinite. Here we provide a bound on the
rate of increase of the resistance, useful mostly when $G$ is of
bounded degree. The proof uses the technique of \cite{trace},
combined with ideas from \cite{path}. For the sake of completeness,
we reproduce the relevant lemmas from \cite{path} and \cite{trace}.

We follow the definitions in \cite{trace}, adapted to finite graphs.
Let $G$ be a finite graph, with two marked vertices, $X_0$ and
$Y_0$. Let $X_i$ be a simple random walk on $G$, started at $X_0$
and stopped when hitting $Y_0$. Let $v(x)$ be the probability of a
simple random walk on $G$, started at $x$, to hit $X_0$ before
$Y_0$. Let $s=\max\{v(y)\dvtx{y \sim Y_0} \}$ and let
$d=\max\{ \deg(x)\dvtx{x\ne Y_0}\}$.

Denote by $C_{\eff}(v \lra u ; H)$ the effective conductance between
$v$ and $u$ in the network $H$. Let $\PATH$ be the subgraph of $G$
consisting of all the edges the random walk crossed before hitting
$Y_0$. We would like to bound the conductance of $\PATH$ from one end
to the other.
\begin{theorem} \label{conductance}
\[
E\bigl(C_{\eff}(X_0 \lra Y_0 ; \PATH)\bigr) \le\frac{12 \log
(d)}{\log(1/s)} .
\]
\end{theorem}

In fact, a stronger form of Theorem \ref{conductance} will be proved, where
the conductance of each edge of $\PATH$ is equal to the number of
times in
which the random walk used that edge.

Recall that the effective resistance is the reciprocal of the
effective conductance. Using the convexity of the function $1/x$ and
Jensen's inequality we get the following corollary.
\begin{corollary}
\[
E\bigl(R_{\eff}(X_0 \lra Y_0 ; \PATH)\bigr) \ge\frac{\log
(1/s)}{12 \log(d)} .
\]
\end{corollary}

We shall now provide the lemmas necessary to proceed with the proof
of Theorem \ref{conductance}. Note that in $\PATH$ the degree of
$Y_0$ is always 1, since the random walk is stopped there.
Therefore, the conductance is always bounded by 1, so the bound is
interesting only when $s$ is small. Hence, we will assume that
$s<1/d$ for the rest of the proof.
\begin{lemma}
If $x$ and $y$ are adjacent vertices of $G\setminus\{Y_0\}$, then
$v(x)\le d v(y)$.
\end{lemma}
\begin{pf}
This follows immediately from the harmonicity of $v$.
\end{pf}

Now, divide the vertices of
$G$ into sets $G_i=\{x\in V(G) \mid d^{-i-1}< v(x) \le d^{-i} \}$.
By the lemma above we get that all the edges in $G$ are within some
$G_i$ or between $G_i$ and $G_{i+1}$ for some $i$.
The following lemma bounds the conductance of
these slices of the graph. This is similar
to \cite{trace}, Lemma 2.3.
\begin{lemma}
\[
C_{\eff}(G_i \lra G_{i+2} ; G) \le2 d^{i+1} C_{\eff}(X_0 \lra
Y_0 ; G) .
\]
\end{lemma}
\begin{pf}
Since $v(X_0)=1$ and $v(Y_0)=0$, the total current flowing through
$G$ is equal to $C_{\eff}(X_0 \lra Y_0 ; G)$. Now, subdivide every
edge $(x,y)$ connecting $G_i$ with $G_{i+1}$, by adding a new vertex
$z$ and replacing the edge $(x,y)$ by edges $(x,z)$ and $(z,y)$
having conductances $c_{xz} = (v(x)-v(y))/(v(x)-d^{-i-1})$ and
$c_{zy} = (v(x)-v(y))/(d^{-i-1} - v(y))$. This subdivision will
result in a network with $v(z)=d^{-i-1}$ and all other voltages
unchanged. Denote the set of new vertices by $Z$. Similarly,
subdividing the edges between $G_{i+1}$ and $G_{i+2}$ yields a new
set $Z'$ of vertices with voltage of $d^{-i-2}$. If
we run current from $G_i$ to $G_{i+2}$ in the modified network
$\widetilde G$,
then all the current must flow through $Z$ and $Z'$. Hence,
$C_{\eff}(G_i \lra G_{i+2};G)\le C_{\eff}(Z \lra Z';\widetilde G)$.
However,
%
%
\begin{equation}
\label{e.Gc}
C_{\eff}(Z \lra
Z';\widetilde G)=\frac{C_{\eff}(X_0 \lra Y_0;G)}{d^{-i-1}-d^{-i-2}} ,
\end{equation}
since the total
current from $X_0$ to $Y_0$ in $\widetilde G$ is $C_{\eff}(X_0 \lra Y_0 ;
G)$ and the voltage difference
between $Z$ and $Z'$ is $d^{-i-1}-d^{-i-2}$. Since $d\ge2$ we get
the required inequality.
\end{pf}

Denote by $N(x,y)$ the number of times the random walk crossed the
edge $(x,y)$, in either direction. Then $\overline G:=(G,E(N))$ is a
new network,
with the same edges as in $G$, but each edge $(x,y)$ has a
conductance equal to the expected number of crossing of $(x,y)$.
\begin{lemma} \label{layer}
\[
C_{\eff}(G_i \lra G_{i+2} ; \overline G) \le4 .
\]
\end{lemma}
\begin{pf}
Let $\widetilde G, Z $ and $Z'$ be as in the proof of the previous lemma.
We use $\widetilde E$ to denote the expectation\vspace*{1pt} with respect to the
random walk on the network $\widetilde G$, and likewise use
$\widetilde C_{xy}$ to denote the conductance of an edge in $\widetilde G$,
etc.
Suppose that an edge $(x,y)$ in $G$ is subdivided in $\widetilde G$
into $(x,z)$ and $(z,y)$.
In that case $E(N(x,y))\le\widetilde E(N(z,y))$,
because the random walk on the
graph $G$ can be coupled with a random walk on the network $\widetilde G$
so that they stay together, except
that the walk on $\widetilde G$ may traverse from $x$ to $z$ and back to
$z$ or
from $y$ to $z$ and back to $y$, while the first random walk stays
in $x$ or $y$, respectively, and similarly for the other subdivided
edges. Let $\overline{\widetilde G}$ be the network whose underlying graph
is that of $\widetilde G$ and where the conductance of every edge is the
expected number of times the random walk on $\widetilde G$ uses that edge.
The above comparison implies that
%
%
\begin{equation}\label{e.GG}
C_{\eff}(G_i \lra G_{i+2} ; \overline G) \le
C_{\eff}(Z\lra Z';\overline{\widetilde G}) .
\end{equation}

Let\vspace*{1pt} $(x,y)$ be an edge of $\widetilde G$ in the part of $\widetilde G$
between $Z$ and $Z'$. We have
$\widetilde E(N(x,y))=\widetilde g(x) \widetilde C_{xy}/\widetilde C_x +
\widetilde g(y) \widetilde C_{xy}/\widetilde C_y$,
where\vspace*{1pt} $g\widetilde(x)$ is the expected number of visits to $x$ before
hitting $Y_0$ and $\widetilde C_x = \sum_{y\sim x} \widetilde C_{xy}$. By
reversibility of
the random walk, we have $\widetilde g(x)/\widetilde C_x = \widetilde
v(x) \widetilde
g(X_0)/\widetilde C_{X_0} $. Since
$\widetilde g(X_0)/\widetilde C_{X_0} = 1/C_{\eff}(X_0 \lra Y_0 ;
{\widetilde G})$
we have
%
%
\begin{equation}
\label{e.NC}
\widetilde E(N(x,y))=
\frac{\widetilde v(x)+\widetilde v(y)}{C_{\eff}(X_0 \lra Y_0 ;
{\widetilde G})}
\widetilde C_{xy}
\le
\frac{2 d^{-i-1}}{C_{\eff}(X_0 \lra
Y_0 ; {\widetilde G})} \widetilde C_{xy} .
\end{equation}
Combining the above estimates, we get
\begin{eqnarray*}
C_{\eff}(G_i \lra G_{i+2} ; \overline G)
&\overset{\mbox{\fontsize{8.36pt}{10.36pt}\selectfont{(\ref
{e.GG})}}}{\le}&
C_{\eff}(Z\lra Z';\overline{\widetilde G})
\\
&\overset{\mbox{\fontsize{8.36pt}{10.36pt}\selectfont{(\ref
{e.NC})}}}{\le}&
C_{\eff}(Z\lra Z';{\widetilde G})
\frac{2 d^{-i-1}}{C_{\eff}(X_0 \lra Y_0 ; {\widetilde G})}
\\
&\overset{\mbox{\fontsize{8.36pt}{10.36pt}\selectfont{(\ref{e.Gc})}}}{=}&
\frac{C_{\eff}(X_0 \lra Y_0;G)}{d^{-i-1}-d^{-i-2}}
\frac{2 d^{-i-1}}{C_{\eff}(X_0 \lra Y_0 ; {\widetilde G})}\\
&=& 2 \frac{d}{d-1}
\le4 .
\end{eqnarray*}

The penultimate equality is valid since the subdivision has no effect
on the effective conductance between $X_0$ and $Y_0$.
\end{pf}

Let $G^N$ denote the network on the graph $G$ where the conductance of
any edge $(x,y)$ is the number of times in which the random walk
path traverses that edge.
Observe that
%
%
\begin{equation} \label{convex}
E\bigl(C_{\eff}(X_0 \lra Y_0 ; G^N)\bigr) \le C_{\eff}(X_0 \lra Y_0
; \overline G )
\end{equation}
follows immediately from the concavity of $C_{\eff}$ (see \cite
{trace} for a proof).

Now, we can complete the proof.
\begin{pf*}{Proof of Theorem \ref{conductance}}
First, notice that since $1 \le N(x,y)$ for every edge $(x,y)\in
\PATH$, we know that $C_{\eff}(X_0 \lra Y_0 ; \PATH) \le C_{\eff}(X_0
\lra Y_0 ; G^N)$. Next, from (\ref{convex}) we get that
$E(C_{\eff}(X_0 \lra Y_0 ; G^N)) \le C_{\eff}(X_0 \lra
Y_0 ;
\overline G)$. To bound this conductance, we note that $C_{\eff}(X_0
\lra Y_0 ; \overline G) \le C_{\eff}(G_0 \lra G_n ; \overline G)$, where
$n=\lfloor\log(1/s) / \log(d) \rfloor$, because, $X_0$ is contained
in $G_0$ and by the definition of $s$, $G_n$ separates $X_0$ from
$Y_0$.

Next, we contract every even $G_i$ to a single vertex. Since
$C_{\eff}(G_i \lra G_{i+2} ; \overline G) \le4$, we have that
\[
C_{\eff}(G_0 \lra G_n ; \overline G) \le\frac{4}{\lfloor n/2 \rfloor}
=\frac{4}{\lfloor q/2\rfloor} ,
\]
where $q= \log(1/s)/ \log(d)$.
If $q\ge12$, this gives
%
%
\begin{equation}\label{e.GN}
E\bigl(C_{\eff}(X_0 \lra Y_0 ; G^N)\bigr) \le\frac{12 \log
d}{\log(1/s)} ,
\end{equation}
while if $q<12$, this holds as well, because the right-hand side is
larger than
$1$ and in $G^N$ the effective conductance between $X_0$ and $Y_0$ is
at most $1$.
This completes the proof.
\end{pf*}

To illustrate the theorem and the estimate (\ref{e.GN}), consider the
two-dimensional lattice
$\ZZ^2$ and the random walk is started at the origin and stopped
upon reaching Euclidean distance larger than some large $r>0$.
We may then contract the vertices of $\ZZ^2$ outside the disk
of radius $r$ to a single vertex $Y_0$. Then $d=4$ and
$s=\Theta((r \log r)^{-1})$, so our bound on the
expected conductance of $\PATH$ is
$O(1/\log r )$. Of course, the conductance in $\ZZ^2$ itself is also
$\Theta(1/\log r )$, and thus the theorem does not
give any new bound in this case. However, the specialization
to this setting of the bound (\ref{e.GN}) is
nontrivial, since a typical edge in $\PATH$ is actually
expected to have a multiplicity of roughly $\log r$.

Perhaps a more interesting example is obtained stopping the walk at
distance~$r$, but considering the expected conductance of $G^N$ or
of $\PATH$ to distance $r/2$.
Here, our
theorem does not apply as is, but it is easy to see that by
choosing $n=\Theta(\log\log r)$ appropriately the above
proof gives a bound of $O(1/\log\log r)$ on the expected conductance.
To appreciate this bound, note that in this case there will
typically be many more edges near the target distance of $r/2$ that are
in $\PATH$.

%

%
\printaddresses

\end{document}